\documentclass[11pt]{amsart}
\usepackage[T1]{fontenc}
\usepackage[utf8]{inputenc}
\usepackage{lmodern}
\usepackage{amsmath,amssymb,amsthm,amsfonts}
\usepackage{mathtools, mathrsfs, verbatim}
\usepackage[left=2.3 cm,right=2.3cm,top=1.9cm,bottom=2.4cm]{geometry}
\usepackage[numbers,sort&compress]{natbib}
\usepackage{hyperref}
\newtheorem{theorem}{Theorem}
\newtheorem{lemma}[theorem]{Lemma}
\newtheorem{corollary}[theorem]{Corollary}

\newtheorem{proposition}[theorem]{Proposition}
\newtheorem{definition}{Definition}

\usepackage[p,osf]{cochineal}
\usepackage[scale=.95,type1]{cabin}
\usepackage[zerostyle=c,scaled=.94]{newtxtt}

\usepackage{color, xcolor}
\definecolor{LinkBlue}{RGB}{45,109,79}
\definecolor{CiteGreen}{RGB}{217,74,42}
\definecolor{UrlPurple}{RGB}{120,0,160}
\PassOptionsToPackage{
  unicode,
  colorlinks=true,
  linkcolor=LinkBlue,
  citecolor=CiteGreen,
  urlcolor=UrlPurple,
  pdfborder={0 0 0}
}{hyperref}
\usepackage{ifpdf}
\ifpdf
    \usepackage[pdftex]{graphicx}
    \DeclareGraphicsExtensions{.png,.pdf,.jpg}
\else
   \usepackage[dvips]{graphicx}
   \DeclareGraphicsExtensions{.eps}
\fi
\graphicspath{{.}{figures/}}
\numberwithin{equation}{section}

\usepackage{doi}
\hypersetup{
  colorlinks=true,
  linkcolor=LinkBlue,
  citecolor=CiteGreen,
  urlcolor=UrlPurple,
  pdfborder={0 0 0}
}
\newcommand{\Ber}{\mathrm{Ber}}
\newcommand{\Pois}{\mathrm{Poi}}
\newcommand{\C}{\mathbb{C}}
\newcommand{\E}{\mathbb{E}}
\newcommand{\Prob}{\mathbb{P}}
\newcommand{\dd}{\,\mathrm{d}}

\newcommand{\tauop}{\tau}
\newcommand{\M}{\mathscr{M}}
\newcommand{\Id}{\mathbf{1}}
\DeclareMathOperator{\diag}{diag}

\newcommand{\yz}[1]{{\color{purple} [Yizhe: #1]}}

\title{The spectral edge of sparse directed Erd\H{o}s--R\'enyi graphs}
\author{Simon Coste}
\address{Laboratoire de Probabilit\'es, Statistique et Mod\'elisation, Universit\'e Paris Cit\'e, Paris, France}
\email{simon.coste@u-paris.fr}
\author{Yizhe Zhu}
\address{Department of Mathematics, University of Southern California}
\email{yizhezhu@usc.edu}
\date{\today}

\begin{document}

\begin{abstract}
Let $d>1$ be fixed and let $A_n$ be an $n\times n$ matrix with independent
$\Ber(d/n)$ entries. For every $0<r<\sqrt d$, we prove that, with high
probability, a positive proportion of the eigenvalues of $A_n$ have modulus
larger than $r$. 
Together with the known upper bound, this implies that the modulus of
the second largest eigenvalue converges in probability to $\sqrt d$.

Our proof works with the Brown measure $\mu_d$ of the adjacency operator of the directed Poisson--Galton--Watson tree. The convergence theorem of Sah, Sahasrabudhe, and Sawhney \cite{sah2023limiting}, together with the Brown-measure identification following it, gives weak convergence of the empirical spectral measure of $A_n$ to $\mu_d$ in probability.
We prove that the outer radius of its support is $\sqrt d$,  using a resolvent recursion on Poisson--Galton--Watson trees.
\end{abstract}

\maketitle

{\footnotesize
\tableofcontents
}

\begin{figure}[h!]
  \centering
  \includegraphics[width=0.49\textwidth]{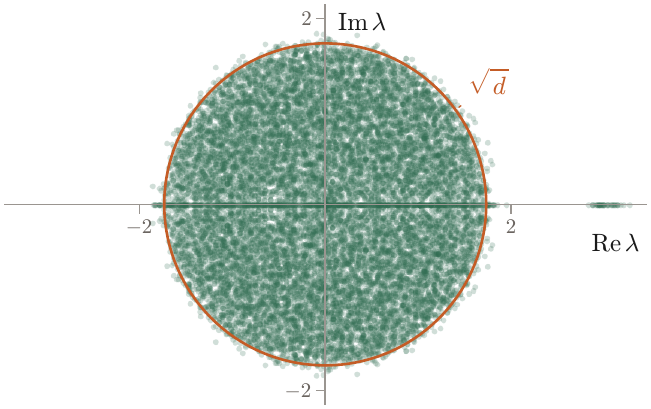}\hfill
  \includegraphics[width=0.49\textwidth]{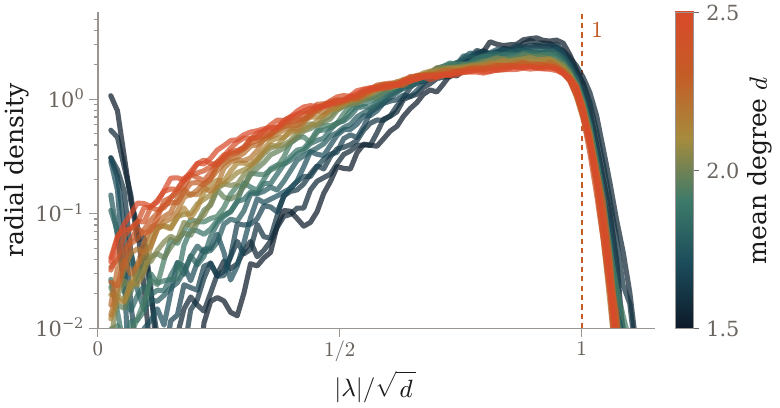}
  \caption{Left: an overlay of the spectra of $N=40$ independent directed Erd\H{o}s--R\'enyi adjacency matrices with $n=500$ vertices and degree $d=3$. The outliers visible near $d$ are the Perron eigenvalues.
  Right: empirical radial densities of $|\lambda|/\sqrt{d}$ for various values of $d\in[1.5,2.5]$, each pooled over $N=50$ independent matrices of size $n=1000$ (Perron eigenvalue and atom at $0$ were removed).  }
  \label{fig:spectral-edge}
\end{figure}

\section{Introduction}

\subsection{Context and main result}

Let $A_n$ be the adjacency matrix of a directed Erd\H{o}s-R\'enyi graph with mean degree $d>1$: all the $n^2$ entries of $A_n$ are independent Bernoulli random variables with mean $d/n$. We note $\lambda_i(A_n)$ its (complex) eigenvalues, ordered by decreasing modulus, and we write
\[
\mu_{A_n}=\frac1n\sum_{j=1}^n\delta_{\lambda_j(A_n)}
\]
for the empirical spectral measure. The main result of \cite{sah2023limiting} is that there is a deterministic probability measure $\mu_d$ on $\C$ such that $\mu_{A_n}$ converges weakly to $\mu_d$ in probability. The same paper identifies $\mu_d$ as the Brown measure of the adjacency operator of the directed Poisson--Galton--Watson tree. It was proved in \cite[Theorem~2.3 and Section~2.2]{bordenave2023detection}
and \cite[Theorem~2.6]{coste2023sparse} that $\lambda_1(A_n)$ converges in
probability to $d$, and that, for every $\varepsilon>0$,
\[
 \Prob\bigl(|\lambda_2(A_n)|>\sqrt d+\varepsilon\bigr)\longrightarrow0.
\]
Consequently,
\begin{equation}\label{eq:outer-containment}
 \operatorname{supp}\mu_d\subseteq\overline D(0,\sqrt d),
\end{equation}
where $\overline D(0,r):=\{z\in\C:|z|\le r\}$ is the closed disk.
It was conjectured in the literature that this upper bound on $|\lambda_2(A_n)|$ is actually optimal, see for example \cite[Section~4]{coste2023sparse} for a discussion. In this paper, we positively answer the question by proving the matching lower bound. 

\begin{theorem}[Outer radius]\label{thm:main}
Let $d>1$, and let $\mu_d$ be the deterministic limiting empirical
spectral measure of matrices with independent $\Ber(d/n)$ entries.
For every $0<r<\sqrt d$,
\begin{equation}\label{eq:main-positive-tail}
 \mu_d\bigl(\{z\in\C:|z|>r\}\bigr)>0.
\end{equation}
In particular, the outer radius of its support is
\begin{equation}\label{eq:outer-radius}
 \max\{|z|:z\in\operatorname{supp}\mu_d\}=\sqrt d.
\end{equation}
\end{theorem}

In view of \eqref{eq:outer-containment}, it suffices to prove that for almost every $r\in(0,\sqrt d)$:
\begin{equation}\label{eq:target-tail}
\mu_d\bigl(\{z\in\C:|z|>r\}\bigr)>0.
\end{equation}
By monotonicity, this then holds for every $r\in(0,\sqrt{d})$.
Fix $0<r<s<\sqrt{d}$, and choose a smooth function
$\varphi:\C\to[0,1]$ that vanishes on $\{|z|\le r\}$ and equals
$1$ on $\{|z|\ge s\}$. Then
\[
c:=\int\varphi\,\dd\mu_d
\ge \mu_d\bigl(\{|z|\ge s\}\bigr)>0.
\]
As a corollary, we show that the modulus of
the second largest eigenvalue of $A_n$ converges in probability to $\sqrt d$. In fact, we can obtain a slightly stronger corollary:
\begin{corollary}[Spectral edge] We have
    \begin{equation}\label{eq:fixed-index-edge}
 |\lambda_j(A_n)|\xrightarrow[n\to\infty]{\Prob}\sqrt d
 \qquad\text{for every fixed }j\ge2.
\end{equation}
\end{corollary}
\begin{proof}
Since $\mu_{A_n}$ converges weakly to $\mu_d$ in probability (see Lemma~\ref{lem:limit-is-brown}), for the function $\varphi$ chosen above,
\[
\Prob\left(\int\varphi\,\dd\mu_{A_n}\ge c/2\right)\longrightarrow 1.
\]
On this event,
\[
\frac1n\#\{j:|\lambda_j(A_n)|>r\}
\ge \int\varphi\,\dd\mu_{A_n}
\ge \frac{c}{2}.
\]
Thus, for every $\varepsilon>0$, there is a positive constant
depending only on $d$ and $\varepsilon$ such that, with probability
tending to one, at least this fraction of the eigenvalues of $A_n$
have modulus greater than $\sqrt d-\varepsilon$. Our proof does not
give a precise estimate for this fraction.

For every fixed integer $j\ge2$, the same argument gives
$|\lambda_j(A_n)|\ge\sqrt d-\varepsilon$ with probability tending to
one. Combining this with
$|\lambda_j(A_n)|\le|\lambda_2(A_n)|$ and the cited upper bound finishes the proof.
\end{proof}

\paragraph{What the theorem says about the support.}
Since $\mu_d$ is rotationally invariant, \eqref{eq:outer-radius} and
closedness of the support imply that the entire circle
$\{z:|z|=\sqrt d\}$ belongs to the support. The origin also belongs to
the support. More precisely,
\begin{equation}\label{eq:zero-atom-lower-bound}
 \mu_d(\{0\})\ge e^{-d}.
\end{equation}
To see this, let $Z_n$ count the zero rows of $A_n$. The row indicators
are independent, each with mean $(1-d/n)^n$, so
$Z_n/n\to e^{-d}$ in probability. The rank is at most $n-Z_n$; hence
the geometric, and therefore the algebraic, multiplicity of the eigenvalue
zero is at least $Z_n$. Thus $\mu_{A_n}(\{0\})\ge Z_n/n$.
Along a subsequence on which both convergences hold almost surely,
the Portmanteau theorem for the closed set $\{0\}$ gives
$\mu_d(\{0\})\ge\limsup_n\mu_{A_n}(\{0\})\ge e^{-d}$.
We have therefore established the support inclusions
\begin{equation}\label{eq:support-inclusions}
 \{0\}\cup\{z:|z|=\sqrt d\}
 \subseteq\operatorname{supp}\mu_d
 \subseteq\overline D(0,\sqrt d).
\end{equation}

These inclusions leave open whether every interior radius occurs in the
support. Full-disk support would require
\begin{equation}\label{eq:no-annular-gaps}
 \mu_d\bigl(\{z:r_1<|z|<r_2\}\bigr)>0
 \qquad(0\le r_1<r_2<\sqrt d).
\end{equation}
This is stronger than \eqref{eq:main-positive-tail}. For example,
a mixture of an atom at zero and the uniform probability measure on the circle of radius $\sqrt d$, with both weights positive, is rotationally invariant and has positive mass outside every smaller disk, but no mass in an interior annulus. Our proof determines the outer radius and does
not establish \eqref{eq:no-annular-gaps}.

\subsection{Related work}

An important result in spectral graph theory is the
\emph{Alon--Boppana bound} \cite{nilli1991second}: for fixed integer
$d\ge3$, every sequence of finite undirected $d$-regular graphs whose
orders tend to infinity satisfies
$\liminf\lambda_2\ge2\sqrt{d-1}$. This lower bound is deterministic.
The matching upper bound for uniformly random regular graphs is the
\emph{Alon--Friedman theorem}; see
\cite{charles2020new,chen2026new,huang2024ramanujan,friedman2008proof,bordenave2019eigenvalues}.
Its proofs are considerably more involved than the elementary lower bound.

When it comes to directed graphs and non-Hermitian matrices, the situation is different. Some tools were developed over the years to prove \emph{upper} bounds for the second eigenvalue; see \cite{coste2021spectral,brito2022spectral,bordenave2022convergence,coste2023sparse,dumitriu2024extreme,coste2024characteristic,belinschi2021outlier}; but proving \emph{lower} bounds for the second eigenvalue of non-Hermitian operators is more difficult. This seems to be due to the lack of min-max characterizations of eigenvalues, and to the very loose relations between singular values and eigenvalues. 

For fixed-degree random regular digraphs, the upper bound
$|\lambda_2|\le\sqrt d+o_{\Prob}(1)$ was proved in
\cite{coste2021spectral,coste2024characteristic}. The recent work
\cite[Theorem~1.2]{he2026orientedkestenmckaylawrandom} proves convergence
to the oriented Kesten--McKay law, whose explicit density has support
$\overline D(0,\sqrt d)$, and hence gives the corresponding lower bound.
This regular model has fixed in- and out-degrees and a different limiting
tree operator from the Poisson-degree model studied here.

For random Erd\H{o}s-Rényi directed graphs, the convergence of the ESD was proved in \cite{sah2023limiting}, but there is no explicit expression of the limiting law $\mu_d$ that could allow for a straightforward study of the spectrum. The proof uses the implicit representations of $\mu_d$ in terms of recursions on random trees.

For the oriented Poisson random graph model, Metz, Neri, and Rogers study these
recursions through the cavity method \cite[Section~4.2]{metz2019spectral}. Their stability calculation predicts the boundary $|z|=\sqrt d$, and their numerical solutions suggest a continuous component filling its interior, together with an atom at zero. These calculations motivate the full-disk question; they do not supply a rigorous proof of \eqref{eq:no-annular-gaps}.  Our theorem establishes the predicted outer radius.

\subsection{Theo McKenzie's work on the non-backtracking matrix}\label{sec:mckenzie}

An important motivation for the study of the spectral edge of random graphs lies in their use for clustering algorithms. 
When clustering random (non-oriented) graphs such as the Stochastic Block Model in the very sparse regime, it turns out that representing the graph with the non-backtracking matrix $B$ gives natural spectral clustering algorithms that work down to the informational limit, known under the name of the Kesten-Stigum threshold; we refer to the papers \cite{krzakala2013spectral, bordenave2015non, abbe2018community} and discussions therein for an overview of this large body of work. 
A key argument in these studies was to prove that the edge of the bulk of the spectrum of $B$ is smaller than $\sqrt{\kappa}$ with $\kappa=\mathbb{E}[D(D-1)]/\mathbb{E}[D]$ and $D$ the root progeny of the limiting tree. This result is similar in flavor to the Alon-Friedman theorem mentioned above. 
The matching lower bound was conjectured in \cite{bordenave2010resolvent}, and Theo McKenzie submitted a solution only a few days before we posted this paper: \cite{mckenzie2026alonboppanaboundnonbacktrackingoperator}. 

Although studying different models and matrices, both papers (ours and \cite{mckenzie2026alonboppanaboundnonbacktrackingoperator}) display the same argument, namely the study of the non-degeneracy of the cavity message $X_\eta$ on the infinite limiting tree; however, the convergence arguments are different, mostly because we can safely work with the Brown measure $\mu_d$, while the non-backtracking spectrum does not converge toward $\mu_d$. 

During the writing of this paper, GPT-5.6 found this core argument, and we spent time reworking and organizing the proof. 
At this moment, the core argument for the non-degeneracy of $X_\eta$ was similar to the one in Lemma 4.2 of \cite{mckenzie2026alonboppanaboundnonbacktrackingoperator}. We used Kesten's martingale limit on the limiting tree to construct a nontrivial "flow" satisfying the branching identity. But on a subsequent pass, GPT-5.6 found a much simpler and elementary proof, which is the one presented here in Section \ref{sec:noncollapse} and which does not rely on the recursion nor on Kesten's martingale, but only on very elementary bounds. For this reason, we think that
\begin{enumerate}
    \item The proof of \cite[Lemma 4.2]{mckenzie2026alonboppanaboundnonbacktrackingoperator} could be simplified as we did; for the (very simple) idea of the core argument, we refer to the paragraph after Proposition \ref{lem:noncollapse}.
    \item But the similarity of \cite[Lemma 4.2]{mckenzie2026alonboppanaboundnonbacktrackingoperator} and our first proof indicates that the "flow construction" found by GPT-5.6 in both problems is probably the most natural one. 
\end{enumerate}

\subsection{Extensions} There is no doubt that our proof method will extend to more complex models of random graphs. We have three extensions in mind: 
\begin{enumerate}
    \item First, studying the edge of the Brown measure of Galton-Watson trees that are more general than Poisson. Here, our proof almost translates verbatim, with the edge being located at $\sqrt{\kappa}$ instead of $\sqrt{d}$, and $\kappa$ is the mean of the size-biased version of the root progeny distribution. 
    \item Second, studying what happens when iid weights are added to the edges of the random graph. Here, the main difference will be that the limiting tree has weights $\xi_{v\to w}$ on the edges, and thus the recursion $X_\eta = V_\eta/(r^2+U_\eta V_\eta)$ remains valid, but with $U_\eta = \eta + \sum_{o \to w}\xi_{o \to w}^2 X^+_{\eta,w}$ instead of \eqref{def:U}, and similarly for $V_\eta$.
    \item We do not show that the full support of $\mu_d$ is $\bar{D}(0, \sqrt{d})$. A full proof would need to show that the cumulative radial density function $F$ is strictly increasing, and for that it is natural to prove that for $r>s$, $F_\eta(r)-F_\eta(s)$ is bounded away from $0$ when $\eta\to 0$ for the regularized version $F_\eta$. 
\end{enumerate}
We leave these extensions for future agentic readers. 

\subsection{Organization of the paper}

Section~\ref{sec:tree-brown-measure} introduces unimodular trees,
the logarithmic potential, the Hermitized resolvent, and the Brown measure
and its regularizations. It also proves rotational invariance of the Brown measure. A key distributional recursion between observables on the tree is proved in Section \ref{sec:recursion}. Then, Section~\ref{sec:edge-criterion} exploits this recursion to give a criterion for the outer edge, and reduces Theorem~\ref{thm:main} to whether a certain observable converges or to zero. 
Section~\ref{sec:noncollapse} studies this convergence and closes the proof of the main Theorem. Section~\ref{sec:cv-esd} explains the identification of the
tree Brown measure with the limiting empirical spectral measure, using
the singular-value estimates of \cite{sah2023limiting}.

\subsection{Acknowledgements and LLM disclosure}

GPT-5.6 assisted the proof development at the level of a coauthor (see also the discussion in Subsection \ref{sec:mckenzie}). This text was written entirely by the authors.  Y.Z. was partially supported by the Simons Grant MPS-TSM-00013944 and NSF DMS-2606337. This work was carried out while S.C. was visiting the Centre de Mathématiques Appliquées at École Polytechnique and while Y.Z. was visiting the Simons Institute for the Theory of Computing during the Spectral Theory Beyond Graphs program in Fall 2026.

\section{The directed tree and its Brown measure}\label{sec:tree-brown-measure}

In this section, we define the objects of interest for the proof of the main theorem.

\begin{definition}[PGW tree]The rooted directed Poisson--Galton--Watson tree of mean
$d$, denoted by $(T,o)$, is defined as follows. The root has independent numbers $N_o^+,N_o^-\sim\Pois(d)$ of outgoing and incoming neighbors. After an edge has been exposed, the vertex $v$ at its
other endpoint has, \emph{in addition to the parent edge}, a number $N_v^- \sim \Pois(d)$ of new incoming neighbors, and a number $N^+_v\sim \Pois(d)$ of outgoing neighbors; all these variables are independent. 
\end{definition}

We note $V$ the set of vertices of $T$. When there is a directed edge from a vertex $v$ toward a vertex $w$, we write $v\to w$. For a realization of the tree $T$, we define the (directed) adjacency operator on finitely supported functions on $V$ by
\[
(A_Tf)(v)=\sum_{w:v\to w}f(w).
\]
For example, $A_T \delta_0(v)$ is equal to 1 if $v$ is an ancestor of $o$, and 0 otherwise. The distribution of $(T,o)$ is \emph{unimodular}: for every nonnegative measurable
function $g(T,x,y)$ of a network with an ordered pair of vertices, it obeys
the mass-transport identity
\[
 \E\sum_{v\in V}g(T,o,v)
 =\E\sum_{v\in V}g(T,v,o).
\]
This is due to the fact that directed PGW trees are the local weak limit of the finite directed Erd\H{o}s--R\'enyi graph rooted at a uniformly chosen vertex; see \cite{aldous2007processes}. 

An \emph{equivariant operator} assigns an operator to each network in a
way that is independent of the root and compatible with network
isomorphisms. More precisely, let $\mu$ denote the law of the unimodular
random rooted network $(T,o)$. A \emph{bounded equivariant operator} is a measurable assignment
\[
  (T,o)\longmapsto B_{T,o}\in\mathscr B(\ell^2(V))
\]
such that $\|B_{T,o}\|\le C$ for $\mu$-almost every $(T,o)$,
where $C<\infty$ is a deterministic constant and $\mathscr B(\ell^2(V))$ denotes the bounded linear operators on $\ell^2(V)$.  We require that the assignment be independent of the root:
\[
  B_{T,o}=B_{T,o'}=:B_T
  \qquad\text{for all }o,o'\in V.
\]
We also require covariance under network isomorphisms: for every isomorphism $\phi:T_1\to T_2$, the unitary operator defined by $U_\phi\delta_v=\delta_{\phi(v)}$ satisfies
\[
  U_\phi B_{T_1}U_\phi^*=B_{T_2}.
\]
Together, these requirements ensure that the operator is independent
of the choice of root and that its matrix coefficients are preserved
under network isomorphisms.

Modulo $\mu$-almost-sure equality, these families of operators form a von Neumann algebra $\M$. Unimodularity implies that the map defined by
\begin{equation}\label{eq:network-trace}
  \tau(B):=\mathbb E_\mu\bigl[\langle B_T\delta_o,\delta_o\rangle\bigr]
\end{equation}
is a faithful normal tracial state on $\M$, simply called the \emph{trace}. 


We now recall the construction of the Brown measure of operators on $(\M,\tauop)$. The reader interested in the technical details will find in the seminal paper \cite{haagerup2007brown} a clear and comprehensive presentation.

\subsubsection*{The determinant class.}
Let $X$ be a closed, densely defined operator affiliated with $\M$.
Its modulus $|X|:=(X^*X)^{1/2}$ is a positive self-adjoint operator.
By the spectral theorem (see, for example, \cite{teschl2014mathematical}), it admits a projection-valued spectral measure $P_{|X|}$ such that
\[
  |X|=\int_{[0,\infty)} t\,\dd P_{|X|}(t).
\]
Since $X$ is affiliated with $\M$, the spectral projections $P_{|X|}(E)$ belong to $\M$. Consequently,
\begin{align}\label{eq:trace_spectral_distribution}
  \mu_{|X|}(E):=\tau\bigl(P_{|X|}(E)\bigr),
  \qquad E\subseteq[0,\infty)\ \text{Borel},
\end{align}
defines a probability measure, called the spectral distribution of $|X|$ with respect to $\tau$.

We say that $X$ belongs to the \emph{determinant class} $\M_\Delta$ of \cite[Definition~2.1]{haagerup2007brown} if
\begin{equation}
  \int_{[0,\infty)} \log^+(t)\,\dd\mu_{|X|}(t)<\infty,
\end{equation}
where $\log^+(t):=\max\{\log t,0\}$ for $t>0$ and $\log^+(0):=0$.
Under this condition, the extended-real integral
\[
  \tau(\log|X|)
  :=\int_{[0,\infty)}\log t\,\dd\mu_{|X|}(t)
  \in[-\infty,\infty)
\]
is well defined, with the convention $\log 0=-\infty$.
The \emph{Fuglede--Kadison determinant} of $X$ is then defined by
\[
  \det(X):=\exp\bigl(\tau(\log|X|)\bigr)\in[0,\infty),
\]
with $\exp(-\infty):=0$. Since $\log^+(t)\le t^2$ for all $t\ge0$,
\[
  \int_{[0,\infty)}\log^+(t)\,\dd\mu_{|X|}(t)
  \le \int_{[0,\infty)}t^2\,\dd\mu_{|X|}(t)
  =\tau(X^*X),
\]
where for a positive affiliated operator $Y$ we use the convention
\[
  \tau(Y):=\int_{[0,\infty)} t\,\dd\mu_Y(t)\in[0,\infty].
\]
Hence every affiliated operator $X$ satisfying
$\tau(X^*X)<\infty$ belongs to the determinant class. In particular, the adjacency operator $A$ of the PGW tree
satisfies
\begin{align}\label{eq:pgw-L2}
  \tau(A^*A)
  =\mathbb E_\mu\bigl[\|A\delta_o\|_2^2\bigr]
  =\mathbb E_\mu[N_o^-]
  =d.
\end{align}
Hence $A$ belongs to the determinant class and has a well-defined Fuglede--Kadison determinant.

\subsubsection*{The Brown measure.}
The determinant class is a vector space, so if $X$ belongs to the
determinant class, then so does $X-z$ for every $z\in\C$. We may
therefore define the logarithmic potential
\[
  L_X(z):=\log\det(X-z)=\tau(\log|X-z|),
  \qquad z\in\C.
\]
Thus $L_X:\C\to[-\infty,\infty)$, where the value $-\infty$ is
allowed. Crucially, $L_X$ is subharmonic and locally integrable on
$\C$; see \cite[Theorem~2.7]{haagerup2007brown}. Its distributional
Laplacian defines a Borel probability measure
\[
  \mu_X:=\frac{1}{2\pi}\Delta L_X,
\]
called the \emph{Brown measure} of $X$, where
$\Delta=\partial_x^2+\partial_y^2$. We emphasize that $\tau$ already
includes the expectation over the rooted network, and hence $\mu_X$
is deterministic. Moreover, the fact that $\mu_X$ has total mass one
is nontrivial; see \cite[Lemma~2.12]{haagerup2007brown}.

An important characterization of the Brown measure is that $\mu_X$ is
the unique probability measure on $\C$ satisfying
\[
  \int_{\C}\log^+|w|\,\dd\mu_X(w)<\infty
\qquad \text{and} \qquad
  L_X(z)
  =\int_{\C}\log|w-z|\,\dd\mu_X(w),
  \qquad z\in\C.
\]

For the adjacency operator $A$ of the directed
Poisson--Galton--Watson tree, we simply write
$L(z):=L_A(z)=\tau(\log|A-z|)$.
We denote its Brown measure by $\mu_d$ rather than $\mu_A$ to
emphasize its dependence on the parameter $d$.

\subsubsection*{The regularization.}
For $\eta>0$, the regularized log-potential is defined as 
\begin{equation}\label{eq:regularized-potential}
L_\eta(z)
=\frac12\tauop(\log\bigl((A-z)(A-z)^*+\eta^2\bigr)),\end{equation}
and the regularized Brown measure as 
\begin{equation}\label{eq:regularized-measure}
\mu_{d,\eta}=\frac1{2\pi}\Delta L_\eta.
\end{equation}
It is explained in \cite{haagerup2007brown} (Lemma 2.8 and proof therein) that $L_\eta$ is a decreasing function of $\eta$, and that $\lim_{\eta \downarrow 0^+}L_\eta(z) = L(z)$. They also prove that $L_\eta$ is locally integrable, so the dominated convergence theorem also shows that 
\begin{equation}\label{eq:L1-convergence}
L_\eta\longrightarrow L
\qquad\text{in }L^1_{\mathrm{loc}}(\C).
\end{equation}
Consequently, for every $\phi\in \mathscr{C}_c^\infty(\C)$,
\[
\int\phi\,\dd\mu_{d,\eta}
=\frac1{2\pi}\int L_\eta\,\Delta\phi\,\dd z
\longrightarrow
\frac1{2\pi}\int L\,\Delta\phi\,\dd z
=\int\phi\,\dd\mu_d.
\]

We finally define the regularized Brown measure,
\begin{equation}\label{eq:regularized-measure}
 \mu_{d,\eta}:=\frac1{2\pi}\Delta L_\eta,
\end{equation}
 where as before the Laplacien is in the sense of of distributions. 
We will recall why it is a probability measure in Section~\ref{sec:edge-criterion}. These facts and \eqref{eq:L1-convergence} will then give weak convergence
to $\mu_d$.



\subsubsection*{Radiality.} We conclude this section with a property that will play a crucial role in what follows. The result is classical and holds for any unimodular random tree. For completeness, we include a proof, but give the details only for $\mu_d$.

\begin{lemma}\label{lem:radial}
    $\mu_d$ and $\mu_{d,\eta}$ are rotationally invariant.
\end{lemma}

\begin{proof}
Fix $\omega\in\C$ with $|\omega|=1$. For each realization of the rooted tree,
construct phases $u_v$ recursively from $u_o=1$. If $v$ is closer to the root
than $w$, set
\[
 u_w=\overline\omega u_v\quad\text{when }v\to w,
 \qquad
 u_w=\omega u_v\quad\text{when }w\to v.
\]
The unique-path property makes this definition consistent and gives
\begin{equation}\label{eq:lem-rotation}
 u_o=1,\qquad u_v\overline{u_w}=\omega
 \quad\text{whenever }v\to w.
\end{equation}

Once we have this vector, we define a diagonal unitary operator by $D\delta_w=u_w\delta_w$ for any $w$. It fixes $\delta_o$. For any other $w$, we have 
\begin{align*}DA_TD^* \delta_w &= \overline{u_v}DA_T\delta_w\\
&= \overline{u_v}D \left(\sum_{v \to w}\delta_w\right)\\
&= \sum_{v \to w}\overline{u_v}u_w \delta_w \\
&= \omega \sum_{v \to w}\delta_w = \omega A_T \delta_w.
\end{align*}
By linearity, this is true for any finitely supported vectors, and then by limiting, it is true for every vector in the domain. Consequently
\[
 D(A_T-z)D^*=\omega(A_T-\overline\omega z),\qquad
 D|A_T-z|D^*=|A_T-\overline\omega z|.
\]
By the spectral functional calculus, the identity above
implies that, for every bounded Borel function
$f:[0,\infty)\to\mathbb R$,
\[
  f(|A_T-\overline\omega z|)
  =D f(|A_T-z|)D^*.
\]
Consequently, their diagonal entries at the root agree:
\begin{align*}
  \bigl\langle
    f(|A_T-\overline\omega z|)\delta_o,\delta_o
  \bigr\rangle
  &=
  \bigl\langle
    D f(|A_T-z|)D^*\delta_o,\delta_o
  \bigr\rangle\\
  &=
  \bigl\langle
    f(|A_T-z|)D^*\delta_o,D^*\delta_o
  \bigr\rangle\\
  &=
  \bigl\langle
    f(|A_T-z|)\delta_o,\delta_o
  \bigr\rangle.
\end{align*}
Taking expectations and using the definition of the trace gives
\[
  \tauop\!\left(f(|A-\overline\omega z|)\right)
  =
  \tauop\!\left(f(|A-z|)\right).
\]
In particular, taking $f=\mathbf1_E$ for every Borel set
$E\subseteq[0,\infty)$ shows that $|A-z|$ and
$|A-\overline\omega z|$ have the same spectral distribution
with respect to $\tauop$.
We now integrate $\log t$ against these identical distributions and obtain
\begin{align*}
  L(z)
  &=\int_{[0,\infty)}\log t\,\mathrm d\mu_{|A-z|}(t)=\int_{[0,\infty)}\log t\,
       \mathrm d\mu_{|A-\overline\omega z|}(t)=L(\overline\omega z).
\end{align*}

Similarly, $D((A_T-z)(A_T-z)^*+\eta^2)D^*=((A_T-\overline\omega z)(A_T-\overline wz)^*+\eta^2)$ and  $L_\eta(z)=L_{\eta}(\overline\omega z)$. Thus the log-potential is radial, and so is its Laplacian, and the Brown measure.
\end{proof}

Since $L(z)$ and $L_\eta(|z|)$ only depend on $|z|$, we define their radial parts $\ell$ and $\ell_\eta$ through
\[
L_\eta(z)=\ell_\eta(|z|), \qquad L(z) = \ell(|z|).
\]
As radial parts of subharmonic functions, these functions $\ell,\ell_\eta$ are differentiable almost everywhere in $r$.

\section{The resolvent recursion on the tree}\label{sec:recursion}

In all this section, we uncover a distributional recursion obeyed by the root resolvent. This recursion will be the key to study the edge of Brown measure of the tree. 

We fix $z\in\C$ with radius $|z|=:r>0$. 
Let $A$ be the directed adjacency operator of $T$, seen as an operator on $\ell^2(V)$ where $V$ is the vertex set of the tree. We will identify $A$ with an infinite matrix, with $A_{u,v} = \mathbf{1}_{u\to v}$. We introduce a copy $V'$ of $V$, and we now consider the Hermitized operator $H=H(z)$ on $\ell^2(V \oplus V')$, defined by 
$$H= \begin{pmatrix}
    0 & A - z \\ A^* - \bar{z} & 0
\end{pmatrix}.$$
We will use the following trick for notations: the additional vertices in $V'$ are just copies of vertices in $V$. Each vertex of the first $V$ will be denoted with letters like $u,v,w…$. The corresponding copies will be denoted with a prime, i.e. $u', v', w'$.  Thus, $u\in V \to u'\in V'$ is a bijection. We will reserve $o$ (respectively $o'$) for the root of $V$ (respectively $V'$). 
The matrix elements of $H$ are given, for any $u,v\in V$, by 
\begin{align*}&H_{u,v'} = \mathbf{1}_{u\to v} - z\mathbf{1}_{u=v},\quad H_{u',v} = \mathbf{1}_{v\to u} - \bar{z}\mathbf{1}_{u=v},\\
&H_{u,v}=0, \quad H_{u',v'}=0.
\end{align*}
We finally define the most important object of this section, the Hermitized resolvent:
\begin{equation}
    G = G_\eta(z) = (H(z) - i \eta)^{-1}.
\end{equation}


Put $B_z=A-z$ and, for $\eta>0$, define the positive self-adjoint operators
\begin{equation}\label{eq:PQ-definition}
 Q_\eta(z)=B_zB_z^*+\eta^2,\qquad
 P_\eta(z)=B_z^*B_z+\eta^2.
\end{equation}
We have a block identity for $G$, which will be used repeatedly in the sequel. 
\begin{lemma}
    \begin{equation}\label{eq:factorisation-G}
 G_\eta(z)=
 \begin{pmatrix}
 i\eta Q_\eta^{-1}&B_zP_\eta^{-1}\\
 B_z^*Q_\eta^{-1}&i\eta P_\eta^{-1}
 \end{pmatrix}.
\end{equation}
\end{lemma}
\begin{proof}
Since $H(z)$ is self-adjoint and $\eta>0$, the scalar identity
\[
 (t-i\eta)^{-1}=\frac{t+i\eta}{t^2+\eta^2},
 \qquad t\in\mathbb R,
\]
gives $
 G_\eta(z)
 =(H(z)+i\eta\Id)(H(z)^2+\eta^2\Id)^{-1}$.
Squaring the block matrix $H(z)$ yields
\[
 H(z)^2+\eta^2\Id
 =
 \begin{pmatrix}
  B_zB_z^*+\eta^2\Id&0\\
  0&B_z^*B_z+\eta^2\Id
 \end{pmatrix}
 =
 \begin{pmatrix}
  Q_\eta&0\\
  0&P_\eta
 \end{pmatrix}.
\]
 Consequently, block multiplication gives
\[
\begin{aligned}
 G_\eta(z)
 &=
 \begin{pmatrix}
  i\eta\Id&B_z\\
  B_z^*&i\eta\Id
 \end{pmatrix}
 \begin{pmatrix}
  Q_\eta^{-1}&0\\
  0&P_\eta^{-1}
 \end{pmatrix}=
 \begin{pmatrix}
  i\eta Q_\eta^{-1}&B_zP_\eta^{-1}\\
  B_z^*Q_\eta^{-1}&i\eta P_\eta^{-1}
 \end{pmatrix}.
\end{aligned}
\]
\end{proof}



From now on fix $|z|=r>0$. If $o\to w$, delete that edge and let $T_w^+$
be the component rooted at $w$. If $y\to o$, define $T_y^-$ similarly.
Conditional on the root degrees, all these rooted branches are independent
directed PGW trees. Write $G_w^+$ and $G_y^-$ for their Hermitized
resolvents at the same $(z,\eta)$.

\begin{proposition}\label{prop:recursion}
The diagonal entries $G_{o,o}$ and $G_{o',o'}$ are purely imaginary and
have the same law. Define
\[
 X_\eta=-iG_{o,o},\qquad
 X_{\eta,w}^+=-i(G_w^+)_{w',w'},\qquad
 X_{\eta,y}^-=-i(G_y^-)_{y,y}.
\]
Then $0<X_\eta\le\eta^{-1}$.
Conditional on the root degrees, the branch messages are independent
copies of $X_\eta$, with a law not depending on these degrees.
Consequently the variables
\begin{align}
 U_\eta&=\eta+\sum_{o\to w}X_{\eta,w}^+,\label{def:U}\\
 V_\eta&=\eta+\sum_{y\to o}X_{\eta,y}^-\label{def:V}
\end{align}
are independent and identically distributed, each a compound Poisson
sum with parameter $d$. They satisfy the pathwise identity
\begin{equation}\label{eq:representation_of_X}
 X_\eta=\frac{V_\eta}{r^2+U_\eta V_\eta}.
\end{equation}
\end{proposition}

\begin{proof}
We first examine the diagonal entries. The factorization
\eqref{eq:factorisation-G} gives
\[
 G_{o,o}=i\eta(Q_\eta^{-1})_{o,o},\qquad
 G_{o',o'}=i\eta(P_\eta^{-1})_{o,o}.
\]
Since $Q_\eta,P_\eta\ge\eta^2\Id$, 
\[
 (Q_\eta^{-1})_{o,o}
 =\langle Q_\eta^{-1}\delta_o,\delta_o\rangle
 =\|Q_\eta^{-1/2}\delta_o\|_2^2\in(0,\eta^{-2}],
\]
and the same argument applies to $P_\eta^{-1}$.  Thus the two root diagonal entries of $G$ are purely
imaginary, and
\[
 0<-iG_{o,o}\le\eta^{-1},\qquad
 0<-iG_{o',o'}\le\eta^{-1}.
\]

Since  edge reversal
preserves the directed PGW law, let \(T^\leftarrow\) be obtained by reversing every edge of \(T\). The directed PGW law is invariant under this operation, and
\[
A_{T^\leftarrow}=A_T^*.
\]
Since \(A_T\) has real coefficients,
\[
\begin{aligned}
Q_{\eta,T^\leftarrow}(z)
&=(A_T^*-zI)(A_T-\overline zI)+\eta^2I=\overline{(A_T^*-\overline zI)(A_T-zI)+\eta^2I}
=\overline{P_{\eta,T}(z)}.
\end{aligned}
\] Taking inverses and then the root diagonal coefficient gives
\[
\bigl(Q_{\eta,T^\leftarrow}(z)^{-1}\bigr)_{o,o}
=\overline{\bigl(P_{\eta,T}(z)^{-1}\bigr)_{o,o}}
=\bigl(P_{\eta,T}(z)^{-1}\bigr)_{o,o},
\]where the last equality holds because a positive operator has real diagonal coefficients.
Therefore, by the block factorization of \(G\),
\[
G_{o,o}(T^\leftarrow,z)
=i\eta\bigl(Q_{\eta,T^\leftarrow}(z)^{-1}\bigr)_{o,o}
=i\eta\bigl(P_{\eta,T}(z)^{-1}\bigr)_{o,o}
=G_{o',o'}(T,z).
\]
Therefore  $G_{o,o}$ and $G_{o',o'}$ have the same law.

Conditional on $N_o^+$ and $N_o^-$, deleting the root leaves independent
directed PGW trees whose laws do not depend on these degrees. Their
resolvents are therefore independent, and the equality of the two
diagonal laws just proved shows that every $X_{\eta,w}^+$ and
$X_{\eta,y}^-$ has the law of $X_\eta$. Moreover, the two root degrees
are independent $\Pois(d)$ variables. The outgoing and incoming
collections are consequently independent, and their sums in
\eqref{def:U}--\eqref{def:V} have the same compound-Poisson law, with the
same added constant $\eta$. 

\medskip

Next we show \eqref{eq:representation_of_X}.  
Order $o,o'$ first and group the remaining coordinates by branch. Then
\[
 H_T(z)-i\eta\Id
 =\begin{pmatrix}J&K\\K^*&L_{\mathrm{br}}\end{pmatrix},
 \qquad
 J=\begin{pmatrix}-i\eta&-z\\-\overline z&-i\eta\end{pmatrix}.
\]
The two rows of $K$ are indexed by $o,o'$, and its columns by all
remaining vertices and their copies. Explicitly, for every $v\ne o$,
\[
 \begin{pmatrix}
 K_{o,v}&K_{o,v'}\\
 K_{o',v}&K_{o',v'}
 \end{pmatrix}
 =\begin{pmatrix}
 0&\mathbf1_{\{o\to v\}}\\
 \mathbf1_{\{v\to o\}}&0
 \end{pmatrix}.
\]
The block $L_{\mathrm{br}}$ is block diagonal, with one block for each
branch $T_b$. Ordering the unprimed coordinates before the primed
coordinates within each branch gives
\[
 L_{\mathrm{br}}
 =\diag_b\!\left[
 \begin{pmatrix}
 -i\eta\Id&A_{T_b}-z\Id\\
 A_{T_b}^*-\overline z\Id&-i\eta\Id
 \end{pmatrix}
 \right].
\]
In particular, all entries between distinct doubled branches are zero.
The block corresponding to $T_b$ is $H_{T_b}(z)-i\eta\Id$, whose
inverse has norm at most $\eta^{-1}$. Therefore
$L_{\mathrm{br}}^{-1}$ is block diagonal with blocks $G_w^+$ and
$G_y^-$, and $\|L_{\mathrm{br}}^{-1}\|\le\eta^{-1}$.

Writing $G_{\mathrm{roots}}$ for the $2\times2$ block of $G$ indexed by
$o,o'$, Schur's formula gives
\[
 G_{\mathrm{roots}}=(J-KL_{\mathrm{br}}^{-1}K^*)^{-1}.
\]

Using these entries of $K$ and the block-diagonal form of
$L_{\mathrm{br}}^{-1}$, we obtain
\begin{align*}
 (KL_{\mathrm{br}}^{-1}K^*)_{o,o}
 &=\sum_{o\to w}\sum_{o\to v}(L_{\mathrm{br}}^{-1})_{w',v'}
 =\sum_{o\to w}(G_w^+)_{w',w'},\\
 (KL_{\mathrm{br}}^{-1}K^*)_{o',o'}
 &=\sum_{y\to o}\sum_{v\to o}(L_{\mathrm{br}}^{-1})_{y,v}
 =\sum_{y\to o}(G_y^-)_{y,y},\\
 (KL_{\mathrm{br}}^{-1}K^*)_{o,o'}
 &=\sum_{o\to w}\sum_{y\to o}(L_{\mathrm{br}}^{-1})_{w',y}=0,\\
 (KL_{\mathrm{br}}^{-1}K^*)_{o',o}
 &=\sum_{y\to o}\sum_{o\to w}(L_{\mathrm{br}}^{-1})_{y,w'}=0.
\end{align*}
In the first two lines, only terms from the same branch survive.
In the last two lines, an outgoing neighbor and an incoming neighbor
belong to distinct branches, since every underlying tree edge has only
one direction. Thus both mixed entries vanish.

By definition,
$(G_w^+)_{w',w'}=iX_{\eta,w}^+$ and
$(G_y^-)_{y,y}=iX_{\eta,y}^-$.
Using \eqref{def:U}--\eqref{def:V}, we obtain
\[
 KL_{\mathrm{br}}^{-1}K^*
 =\begin{pmatrix}i(U_\eta-\eta)&0\\0&i(V_\eta-\eta)\end{pmatrix}.
\]
The $2\times2$ matrix inversion formula therefore gives
\begin{equation}\label{eq:2x2inversion}
 G_{\mathrm{roots}}
 =\begin{pmatrix}-iU_\eta&-z\\-\overline z&-iV_\eta\end{pmatrix}^{-1}
 =\frac1{U_\eta V_\eta+|z|^2}
   \begin{pmatrix}iV_\eta&-z\\-\overline z&iU_\eta\end{pmatrix}.
\end{equation}
Taking its top-left entry and multiplying by $-i$ yields
\[
 X_\eta=-iG_{o,o}=\frac{V_\eta}{r^2+U_\eta V_\eta},
\]
which is \eqref{eq:representation_of_X}.
Empty sums are zero, so the same formulas hold when either root degree
is zero. If the root has no neighbors, the branch correction vanishes
and $U_\eta=V_\eta=\eta$.
\end{proof}

\section{The edge criterion}\label{sec:edge-criterion}

We now define the cumulative radial distribution functions,
\begin{align*}
&F_\eta(r)=\mu_{d,\eta}(\{z:|z|\le r\}), \quad 
F(r) = \mu_d(\{z:|z|\le r\}).
\end{align*}
The key to the proof of the main theorem is the following representation of $F_\eta$ in terms of the distribution of $X_\eta$. We recall that $X_\eta = (H(z) - i\eta)^{-1}_{o,o}$. 

\begin{lemma}\label{lem:radial-cdf}
For every $\eta>0$ and $r>0$, with  $|z|=r$,
\begin{equation}\label{eq:radial-cdf}
 F_\eta(r)
 =\E\left[\frac{r^2}{r^2+U_\eta V_\eta}\right]
 =1-\E\left[\frac{U_\eta V_\eta}{r^2+U_\eta V_\eta}\right].
\end{equation}
In particular, $\mu_{d,\eta}$ is a probability measure.
\end{lemma}

\begin{proof}[Proof of \eqref{eq:radial-cdf}] 
By Green's formula and the radiality of $L_\eta$, we have
\begin{equation}\label{eq:green-radial}
F_\eta(r)
=\frac1{2\pi}\int_{|z|\le r}\Delta L_\eta(z)\,\dd z
=\frac1{2\pi}\int_{|z|=r}\partial_nL_\eta(z)\,\dd s
=r\ell_\eta'(r)
\end{equation}
where $\partial_n$ is the normal derivative. This is valid only for radii $r$ at which $\ell_\eta$ is differentiable, that is, for almost every $r>0$. We now compute $\ell_\eta'(r)$. Let us note 
$$h_\eta=(A-z)(A-z)^*+\eta^2,$$ 
so that by definition (see \eqref{eq:regularized-potential}) $L_\eta = (1/2)\tau(\log h_\eta)$. With the convention
$\partial_z=\frac12(\partial_x-i\partial_y)$ and a simple use of the chain rule, it can be proved that
\begin{align*}
\partial_zL_\eta(z)&=\frac{1}{2}\tauop\bigl((\partial_z h_\eta )h_\eta^{-1}\bigr)\\&=-\frac{1}{2}\tauop\bigl((A-z)^*h_\eta^{-1}\bigr)\\
&= -\frac{1}{2}\mathbb{E}\left[\left\langle (A-z)^* h_\eta^{-1}\delta_o, \delta_o \right\rangle\right]
\end{align*}
see for instance the discussion above Equation (2.19) in \cite{haagerup2007brown}. 

As a consequence of \eqref{eq:factorisation-G} together with \eqref{eq:2x2inversion}, we get
\begin{align*}
\langle (A-z)^* h_\eta^{-1}\delta_o, \delta_o \rangle &= \langle (H-i\eta)^{-1}\delta_{o’},\delta_o \rangle \\&=-\frac{\bar{z}}{U_\eta V_{\eta} + r^2}
\end{align*}
and thus
\begin{equation}\label{eq:potential-derivative-cavity}
\partial_zL_\eta(z)
=\frac{\overline z}{2}\E\left[\frac1{r^2+U_\eta V_\eta}\right].
\end{equation}
On the other hand, the fact that $L_\eta$ is radial gives
$\partial_zL_\eta(z)=\frac{1}{2r}\overline z\,\ell_\eta'(r)$. Since $\bar{z}\neq 0$, comparing this with \eqref{eq:potential-derivative-cavity} gives 
\[
\ell_\eta'(r)
=\E\left[\frac{r}{r^2+U_\eta V_\eta}\right].
\]
At this point, we directly see that this whole expression does depend on $z$ only through $|z|=r$. Moreover, combining this with \eqref{eq:green-radial} and taking the complement $1-F_\eta(r)$ yields \eqref{eq:radial-cdf}.

    Finally, Proposition~\ref{prop:recursion} implies
$U_\eta\le\eta+N_o^+/\eta$ and
$V_\eta\le\eta+N_o^-/\eta$. Their independence yields, uniformly in $z$,
\[
 0\le1-F_\eta(r)
 \le\frac{\E[U_\eta V_\eta]}{r^2}
 \le\frac{(\eta+d/\eta)^2}{r^2}.
\]
For fixed $\eta$, the last bound tends to zero as $r\to\infty$.
Continuity from below of the measure proves
$\mu_{d,\eta}(\C)=\lim_{r\to\infty}F_\eta(r)=1$.
\end{proof}

Equipped with this representation, we can now better understand the support of the measure $\mu_d$ in terms of the limiting distribution of $X_\eta$. 

\begin{theorem}At almost any continuity point $r$ of $F$, we have the equivalence
    \begin{equation}
    \label{eq:main-equivalence}
    F(r)=1 \quad \Longleftrightarrow \quad X_\eta\xrightarrow[]{} 0 \text{ in probability as $\eta\downarrow0$ for some } |z|=r.
\end{equation}
\end{theorem}

\begin{proof}Let us start by an important fact, namely that if $X_\eta$ goes to zero in probability, then so do $U_\eta, V_\eta$, and vice-versa. The law of $X_\eta(z)$ depends only on $|z|$, by
Lemma~\ref{lem:radial} and \eqref{eq:factorisation-G}.
Write $U_\eta=\eta+\sum_{j=1}^N X^+_{\eta,j}$ in distribution,
where $N\sim\Pois(d)$ is independent of the i.i.d.\ branch messages $X_{\eta,j}^+$. Let us fix some $\delta>0$. 
For $\eta<\delta/2$ and any integer $M\ge1$,
\begin{equation}\label{eq:compound-poisson-tail}
 \Prob(U_\eta>\delta)
 \le\Prob(N>M)+M\Prob\left(X_\eta>\frac{\delta}{2M}\right).
\end{equation}
Indeed, when $N\le M$, a sum exceeding $\delta/2$ must contain
a term exceeding $\delta/(2M)$.
Thus $X_\eta\to0$ in probability implies $U_\eta,V_\eta\to0$ in
probability, by first fixing $M$ and sending $\eta$ to 0, then letting $M\to\infty$. Conversely, if $U_\eta$ goes to zero in probability, then the inequality $0 \le X_\eta \le V_\eta/r^2$ ensures that $X_\eta$ also goes to zero in probability.

 Let us prove the $\Leftarrow$ implication. We consider the event where both $U_\eta$ and $V_\eta$ are smaller than $\delta$.
On this event, $U_\eta V_\eta\le\delta^2$, so
$U_\eta V_\eta/(r^2+U_\eta V_\eta)\le\delta^2/(r^2+\delta^2)$;
on its complement, the fraction is smaller than 1. Hence
\[
 0\le1-F_\eta(r)
 =\E\left[\frac{U_\eta V_\eta}{r^2+U_\eta V_\eta}\right]
 \le\frac{\delta^2}{r^2+\delta^2}
      +\Prob(U_\eta>\delta)+\Prob(V_\eta>\delta).
\]
First let $\eta\downarrow0$ with $\delta$ fixed, so that the two probabilities tend to 0; then let $\delta\downarrow0$ and the remaining term $\delta^2/(r^2+\delta^2)$ goes to zero, thus proving that $F_\eta(r)\to1$. 
Finally, $r$ is a continuity point of $F$, so
weak convergence of $\mu_{d,\eta}$ also gives $F_\eta(r)\to F(r)$.
Consequently $F(r)=1$.

Let us finally prove the $\Rightarrow$ implication. Suppose $F(r)=1$ and fix $\delta>0$.
Since $r$ is a continuity point of $F$, the weak convergence of $\mu_{d,\eta}$ to $\mu_d$ gives $F_\eta(r)\to F(r)=1$.
On the event $\{U_\eta>\delta,V_\eta>\delta\}$ we have
$U_\eta V_\eta>\delta^2$ and $U_\eta V_\eta / (r^2+U_\eta V_\eta) \ge \delta^2/(r^2+\delta^2)$. Hence, pointwise,
\[
 \frac{U_\eta V_\eta}{r^2+U_\eta V_\eta}
 \ge\frac{\delta^2}{r^2+\delta^2}
       \mathbf1_{\{U_\eta>\delta,V_\eta>\delta\}}.
\]
Taking expectations and using \eqref{eq:radial-cdf}, we obtain
\[
 1-F_\eta(r)
 =\E\left[\frac{U_\eta V_\eta}{r^2+U_\eta V_\eta}\right]
 \ge\frac{\delta^2}{r^2+\delta^2}
       \Prob(U_\eta>\delta,V_\eta>\delta).
\]
 Rearranging the terms and using the independence between $U_\eta$ and $V_\eta$ gives
\[
 \Prob(V_\eta>\delta)^2
 =\Prob(U_\eta>\delta,V_\eta>\delta)
 \le\frac{r^2+\delta^2}{\delta^2}\bigl(1-F_\eta(r)\bigr)
 \longrightarrow0.
\]
Thus $V_\eta$ goes to zero in probability, and so does $X_\eta$. 
\end{proof}

Thanks to this criteron, we only have to study the convergence in probability of $X_\eta$. This is done in the next Proposition, which is the crucial argument of this paper. 

\begin{proposition}\label{lem:noncollapse}
$X_\eta$ does not converge in probability to zero when $\eta \downarrow 0$ if $|z|=r<\sqrt{d}$. 
\end{proposition}

Let us give an informal exposition of the proof, which goes by contradiction. The entry point is the representation
$$X_\eta = \frac{V_\eta}{r^2+U_\eta V_\eta}.$$
Let us suppose that $X_\eta \to 0$ in probability. We saw earlier that this implies that $U_\eta \to 0$ in probability when $\eta \to 0$. But then, we should have (informally) that $X_\eta \sim V_\eta /r^2$. Now, by definition, $V_\eta =\eta+\sum_{j=1}^N X^-_{\eta,j} \ge \sum_{j=1}^N X^-_{\eta,j}$, where the $X_{\eta,j}^-$ are iid copies of $X_\eta$ and $N$ is a Poisson random variable with mean $d$. If we could take expectations, we would thus get 
$$\E[X_\eta] \gtrsim \frac{1}{r^2}\E\left[\sum_{i=1}^N X^-_{\eta,j}\right] = \frac{d}{r^2}\E[X_\eta],$$
which leads to a contradiction for every $r < \sqrt{d}$. The only technical challenge here is a quantitative control on the approximation $X_\eta \sim V_\eta/r^2$, which is rigorously worked out in the full proof of Proposition \ref{lem:noncollapse} in Section \ref{sec:noncollapse}. Once the Proposition is accepted, we are in position to close the proof of the main theorem. 

\begin{proof}[Proof of Theorem~\ref{thm:main}]
    In the equivalence \eqref{eq:main-equivalence}, the restriction on $r$ being a continuity point for $F$ is only apparent. Indeed, $F$ is an increasing and measurable function, and thus only has a countable number of discontinuities. For any $r<\sqrt{d}$, one can always find a continuity point $s \in (r,\sqrt{d})$ for which the preceding Proposition shows that $X_{\eta}$ does not go to zero in probability, and thus $F(r)\leqslant F(s)<1$. 
\end{proof}

\section{Proof of the non-degenerescence of the cavity message $X_\eta$}\label{sec:noncollapse}

We now prove Proposition \ref{lem:noncollapse}.
Fix $z\in\C$ with $0<|z|=r<\sqrt d$. We will prove the following stronger statement: 
\begin{equation}\label{eq:prop-target-0}
 \text{there exist }\delta,\beta,\eta_0>0\text{ such that }
 \Prob(X_\eta\ge\delta)\ge\beta
 \quad\text{for all }0<\eta<\eta_0.
\end{equation}
We start by choosing a number $\alpha$ satisfying
\[
 d^{-1}<\alpha<\min\{1,r^{-2}\}.\]
 This is possible because $r^2<d$ and $1<d$. Then, we set
\[
 u_\eta:=\E[\min\{1,X_\eta\}].
\]
We saw that $X_\eta>0$ almost surely, hence $u_\eta>0$. The key observation of the proof is the following fact: on the event $\{U_\eta\le1-\alpha r^2\}$, we have
\begin{equation}\label{eq:capped-message}
 X_\eta \ge\min\{1,\alpha V_\eta\}.
\end{equation}
To see this, recall that $X_\eta=\frac{V_\eta}{r^2+U_\eta V_\eta}$. If $V_\eta\le1/\alpha$, the denominator is at most
$r^2+(1-\alpha r^2)/\alpha=1/\alpha$ and thus $X_\eta \ge \alpha V_\eta$. 
If $V_\eta\ge1/\alpha$, then $X_\eta \geq \alpha^{-1}/(r^2+(1-\alpha r^2)\alpha^{-1})=1$, because the function
$v \mapsto v/(r^2+(1-\alpha r^2)v)$ is increasing. 

Write $V_\eta=\eta+\sum_{j=1}^N X^-_{\eta,j}$, where
$N\sim\Pois(d)$ and the $X^-_{\eta,j}$ (who represent the incoming branch messages) are iid copies of $X_\eta$, and are all independent
of $U_\eta$ thanks to Proposition~\ref{prop:recursion}. Then, dropping the $\eta$ in the definition of $V_\eta$, \eqref{eq:capped-mean} gives the recursion
\begin{align*}X_\eta &\ge  \min \left\lbrace 1, \sum_{j=1}^N \alpha X^-_{\eta,j}\right\rbrace. \end{align*}
From there, we have
\begin{align*}
\min\{1,X_\eta\}&\ge  \min \left\lbrace 1, \sum_{j=1}^N \min\{1, \alpha X^-_{\eta,j}\}\right\rbrace.
\end{align*}
We are going to use the following elementary inequality: for any $t_1,\ldots,t_N\in[0,1]$,
\[
 1-\prod_{j=1}^N(1-t_j)
 \le\min\left\{1,\sum_{j=1}^N t_j\right\},
\]
with the convention that an empty sum is 0 and an empty product is 1. 
That the LHS is smaller than 1 is immediate; the bound by the sum follows
can be proved using an induction on $N$ and the inequality $s+t-st\le s+t$. 

We apply this inequality with $t_j=\alpha\min\{1,X^-_{\eta,j}\}\in[0,1]$ on the event $\{U_\eta \le 1-\alpha r^2\}$:
\[
 \min\{1,X_\eta\}\ge
 \mathbf1_{\{U_\eta\le1-\alpha r^2\}}
 \left[1-\prod_{j=1}^N
       \bigl(1-\alpha\min\{1,X^-_{\eta,j}\}\bigr)\right].
\]
We now take expectations, condition on $N$, and use the Poisson generating function $\E[s^N] = e^{d(s-1)}$. Note that $U_\eta$ is entirely independent from $N$ or the $X^-_{\eta,j}$, so we get
\begin{align}
 u_\eta
 &\ge\Prob(U_\eta\le1-\alpha r^2)
       \left(1-\E[(1-\alpha u_\eta)^N]\right)=\Prob(U_\eta\le1-\alpha r^2)
       \bigl(1-e^{-d\alpha u_\eta}\bigr).\label{eq:capped-mean}
\end{align}

We are now equipped to show \eqref{eq:prop-target-0}, by contradiction. Suppose indeed that \eqref{eq:prop-target-0} were false. We could choose
$\eta_k\downarrow0$ with
$\Prob(X_{\eta_k}\ge1/k)<1/k$, and hence
$X_{\eta_k}\to0$ in probability, which would imply that (1) $u_{\eta_k}\to0$ and (2) $U_{\eta_k}\to0$ in probability, as we saw just after \eqref{eq:compound-poisson-tail}. But remember that we chose $\alpha$ so that $1-\alpha r^2>0$, hence 
$$\lim_{k \to \infty}\Prob(U_{\eta_k} \le 1-\alpha r^2) = 1.$$
But then, \eqref{eq:capped-mean} would imply that 
$$1 \geq \lim_{k \to \infty}\frac{1-e^{-d\alpha u_{\eta_k}}}{u_{\eta_k}} = d\alpha >1,$$
This is the sought contradiction.

\section{Convergence of the empirical spectral measure}\label{sec:cv-esd}
Theorem~1.1 of \cite{sah2023limiting} states that the empirical spectral
measure converges in probability to a deterministic measure.  The discussion in \cite{sah2023limiting} following that theorem identifies the limit with the Brown measure of the directed Poisson--Galton--Watson adjacency operator.  This section, which is mostly expository, supplies some details regarding this identification. 

\begin{lemma}[Identification of the limiting measure]
\label{lem:limit-is-brown}
Let $d>1$, let $A_n$ have independent $\Ber(d/n)$ entries, and let $A$ be the closed
directed adjacency operator defined above, regarded as an operator affiliated
with $(\M,\tauop)$.  Then $\mu_A$ is a deterministic probability measure and
\[
  \mu_{A_n}
  \xrightarrow[n\to\infty]{}
  \mu_A
\]
where the convergence is weak in probability, which means that for every $f\in \mathscr{C}_b(\C)$ and every $\varepsilon>0$,
\[
 \Prob\left(
   \left|\int f\,\dd\mu_{A_n}-\int f\,\dd\mu_A\right|>\varepsilon
 \right)
 \longrightarrow0.
\]
Consequently, the deterministic measure $\mu_d$ in
\cite[Theorem~1.1]{sah2023limiting} is $\mu_A$. 
\end{lemma}

\begin{proof}
By \eqref{eq:pgw-L2}, $\tauop(A^*A)=d<\infty$. Thus the construction
of the Brown measure and \cite[Theorem~2.7]{haagerup2007brown} gives a probability measure
$\mu_A=(2\pi)^{-1}\Delta L$, where
\[
 L(z)=\int_0^\infty\log t\,\dd\nu_{A,z}(t),
 \qquad \nu_{A,z}:=\mu_{|A-z|}.
\]
Both are deterministic because $\tauop$ includes an expectation over
the whole distribution of the Poisson-Galton-Watson tree. We now explain the convergence towards $\mu_A$, which has two main steps.

\bigskip

\paragraph{Step 1: identification of the shifted singular laws.}
Fix $z\in\C$ and write
\[
 X_n=A_n-z\Id,
 \qquad \nu_{n,z}=\frac1n\sum_{j=1}^n\delta_{\sigma_j(X_n)}.
\]
By \cite[Lemmas~11.1 and~4.1; see also the proof of
Theorem~1.1, p.~66]{sah2023limiting},
$\nu_{n,z}$ converges weakly in probability to a deterministic
probability measure $\nu_z$ on $[0,\infty)$.
Lemma~4.1 transfers the singular-value convergence from the auxiliary
model in that paper to the iid model: the total-variation error tends
to zero, and a simultaneous permutation of rows and columns preserves
the singular values of the shifted matrix.
The proof of \cite[Lemma~11.1, p.~53]{sah2023limiting} also establishes
that $\nu_z$ is uniquely determined by its even moments among probability
measures on $[0,\infty)$. We will use this uniqueness conclusion after
matching those moments with the moments of $\nu_{A,z}$.
Passing from weak convergence to moments requires care because the
test functions $t^{2k}$ are unbounded; we justify this passage below.

For $k\ge1$, put
\begin{equation}\label{eq:finite-shifted-moments}
 M_{n,k}(z):=\int t^{2k}\,\dd\nu_{n,z}(t)
 =\frac1n\operatorname{Tr}\big[(X_n^*X_n)^k\big].
\end{equation}
Write $a_{uv}=(A_n)_{uv}$. Since these entries are real,
\[
 (X_n^*)_{uv}=a_{vu}-\overline z\,\mathbf1_{\{u=v\}},
 \qquad (X_n)_{uv}=a_{uv}-z\,\mathbf1_{\{u=v\}}.
\]
Expanding the matrix product in \eqref{eq:finite-shifted-moments}
and summing its diagonal entries gives a cyclic sequence of indices
$i_0,i_1,\ldots,i_{2k}=i_0$. The odd steps use $X_n^*$ and the even
steps use $X_n$. Thus
\begin{equation}\label{eq:shifted-index-expansion}
\begin{aligned}
 M_{n,k}(z)
 ={}&\frac1n\sum_{i_0,\ldots,i_{2k-1}\in[n]}
 \prod_{\ell=1}^k
 \Big[\big(a_{i_{2\ell-1},i_{2\ell-2}}
       -\overline z\,\mathbf1_{\{i_{2\ell-2}=i_{2\ell-1}\}}\big)\big(a_{i_{2\ell-1},i_{2\ell}}
       -z\,\mathbf1_{\{i_{2\ell-1}=i_{2\ell}\}}\big)\Big].
\end{aligned}
\end{equation}
At each position, select either the random entry or the scalar term.
A scalar choice forces consecutive indices to agree and contributes
$-\overline z$ at odd positions and $-z$ at even positions.
A random choice uses the directed edge $(i_j,i_{j-1})$ at an odd
position $j$ and $(i_{j-1},i_j)$ at an even position.
Group these terms by their choices and by the equality relations
among the indices, relabeling vertices in order of first appearance.
Call each resulting pattern $\pi$, and write $v(\pi)$ for its number
of vertices, $e(\pi)$ for its number of distinct ordered random edges (including self-loops),
and $c_\pi(z)$ for the product of its scalar factors, with the empty
product equal to $1$. Repeated uses of the same ordered entry count
only once in $e(\pi)$; the entries $a_{uv}$ and $a_{vu}$ are distinct
when $u\ne v$.

There are $(n)_{v(\pi)}$ injective labelings of the pattern, where
$(n)_v=n(n-1)\cdots(n-v+1)$; this number is zero if $v>n$.
For any such labeling, the expectation of the product of its random
entries is $(d/n)^{e(\pi)}$: indeed, $a_{uv}^r=a_{uv}$ for $r\ge1$,
and the distinct ordered entries are independent with mean $d/n$.
Multiplying by the scalar factors, the number of labelings, and the
trace normalization $1/n$, we find that the pattern contributes
\[
 c_\pi(z)\,\frac{(n)_{v(\pi)}}n
             \left(\frac dn\right)^{e(\pi)}.
\]
There are finitely many patterns for fixed $k$: there are only
$2^{2k}$ choices of random or scalar terms, and finitely many equality
patterns among the $2k$ index positions. We can therefore take the
limit separately for each pattern.
Every new vertex is reached by a random step, since a scalar step
stays at the same vertex. Hence the random edges connect all the
visited vertices, and $e(\pi)\ge v(\pi)-1$.
Equality holds precisely when the underlying graph is a tree and each
edge has just one orientation. A loop or a second orientation adds an
ordered edge without helping to connect new vertices, while an
undirected cycle already requires at least as many edges as vertices.
For each fixed pattern,
\[
 \frac{(n)_{v(\pi)}}n\left(\frac dn\right)^{e(\pi)}
 =d^{e(\pi)}\frac{(n)_{v(\pi)}}{n^{v(\pi)}}
          n^{v(\pi)-1-e(\pi)},
 \qquad
 \frac{(n)_{v(\pi)}}{n^{v(\pi)}}\longrightarrow1.
\]
Its contribution therefore vanishes if $e(\pi)>v(\pi)-1$ and
converges to $c_\pi(z)d^{e(\pi)}$ otherwise. Consequently,
\begin{equation}\label{eq:shifted-tree-pattern-sum}
 \E M_{n,k}(z)\longrightarrow
 m_{2k}(z):=\sum_{\pi\text{ a tree pattern}}
                 c_\pi(z)d^{e(\pi)}.
\end{equation}
The all-scalar pattern is included as the one-vertex tree with no edges. The scalar choices do not contribute to $e(\pi)$.

We next justify passing these expected moments to the weak limit
$\nu_z$. First, the same finite pattern expansion gives a uniform
bound at every fixed order. Indeed, for $n\ge v(\pi)$,
\[
 |c_\pi(z)|\,\frac{(n)_{v(\pi)}}n
       \left(\frac dn\right)^{e(\pi)}
 \le |c_\pi(z)|d^{e(\pi)}n^{v(\pi)-1-e(\pi)}
 \le |c_\pi(z)|d^{e(\pi)},
\]
and the contribution is zero for $n<v(\pi)$. Summing these bounds
over the finitely many patterns at order $k+1$ gives
\begin{equation}\label{eq:shifted-higher-moment-bound}
 \sup_n\E M_{n,k+1}(z)<\infty.
\end{equation}

Observe that, on $t>R>0$,
$t^{2k}=t^{2k+2}/t^2\le R^{-2}t^{2k+2}$. Hence
\begin{equation}\label{eq:shifted-moment-tail}
\begin{aligned}
 \sup_n\E\int_{t>R}t^{2k}\,\dd\nu_{n,z}(t)
 &\le \frac1{R^2}\sup_n\E\int_{t>R}t^{2k+2}\,\dd\nu_{n,z}(t)\\
 &\le \frac1{R^2}\sup_n\E M_{n,k+1}(z)
   \xrightarrow[R\to\infty]{}0.
\end{aligned}
\end{equation}
The last limit follows from \eqref{eq:shifted-higher-moment-bound};
in particular, the tail is small uniformly in $n$.
For fixed $R$, the function $\min\{t^{2k},R^{2k}\}$ is bounded and
continuous. Moreover,
\[
 0\le t^{2k}-\min\{t^{2k},R^{2k}\}
 \le t^{2k}\mathbf1_{\{t>R\}}.
\]
Integrating and using \eqref{eq:shifted-moment-tail} gives
\[
 0\le \E M_{n,k}(z)
     -\E\int\min\{t^{2k},R^{2k}\}\,\dd\nu_{n,z}(t)
 \le\frac1{R^2}\sup_n\E M_{n,k+1}(z).
\]
Now take $n\to\infty$, using \eqref{eq:shifted-tree-pattern-sum}
for the first term and bounded continuous convergence for the second:
\[
 0\le m_{2k}(z)
     -\int\min\{t^{2k},R^{2k}\}\,\dd\nu_z(t)
 \le\frac1{R^2}\sup_n\E M_{n,k+1}(z).
\]
Finally, as $R\to\infty$, the truncated integrands increase to
$t^{2k}$ and the right-hand side tends to zero. Monotone convergence
therefore proves that
\begin{equation}\label{eq:shifted-limit-moment}
 \int t^{2k}\,\dd\nu_z(t)=m_{2k}(z)<\infty.
\end{equation}

On the tree, the root-trace formula \eqref{eq:network-trace} gives
\begin{equation}\label{eq:tree-shifted-moment}
 \int t^{2k}\,\dd\nu_{A,z}(t)
 =\E\big\langle |A_T-z\Id|^{2k}\delta_o,\delta_o\big\rangle.
\end{equation}
We can thus expand it as in
\eqref{eq:shifted-index-expansion}, now fixing $i_0=i_{2k}=o$,
omitting $1/n$, and using $a_{uv}=\mathbf1_{\{u\to v\}}$.
Only tree patterns occur, since every connected subgraph of the
directed PGW tree has a tree as its underlying graph, with no loops
or pairs of opposite edges. The scalar choices have the same weights
as in the finite-matrix expansion.

To count a pattern's rooted embeddings, root it at its initial vertex.
If a pattern vertex has $r^+$ outgoing and $r^-$ incoming children,
excluding its parent, their distinct labels require ordered choices
of new neighbors. For independent $N^+,N^-\sim\Pois(d)$, there are
$(N^+)_{r^+}(N^-)_{r^-}$ such choices. They are ordered because each
child occupies a specified position in the index pattern.
For $N\sim\Pois(d)$, the factorial-moment identity gives
\[
 \E(N)_r=\sum_{m=r}^\infty
          \frac{m!}{(m-r)!}e^{-d}\frac{d^m}{m!}
 =d^r e^{-d}\sum_{j=0}^\infty\frac{d^j}{j!}=d^r,
 \qquad r\ge0.
\]
For $r=0$, the falling factorial is $1$, also when $N=0$.
Independence of the incoming and outgoing offspring counts therefore
gives $\E[(N^+)_{r^+}(N^-)_{r^-}]=d^{r^++r^-}$.
Conditional on the selected children, their descendant trees are
independent copies of the original rooted tree. Repeating the count
at each child, and then down the finite pattern, multiplies one factor
$d$ for each edge. 
It follows that the expected number of rooted embeddings is
$d^{e(\pi)}$. 

There are only finitely many patterns.
Consequently, we have
\[
 \int t^{2k}\,\dd\nu_{A,z}(t)=m_{2k}(z)
 =\int t^{2k}\,\dd\nu_z(t),\qquad k\ge1.
\]

We can now apply the uniqueness conclusion from the proof of
\cite[Lemma~11.1, p.~53]{sah2023limiting}. To explain precisely why
it applies, that proof establishes uniqueness of the symmetrized
probability measure on $\mathbb R$ with the prescribed even moments
and zero odd moments. Symmetrizing a measure on $[0,\infty)$ means
averaging its images under $t\mapsto t$ and $t\mapsto-t$.
The symmetrizations of $\nu_{A,z}$ and $\nu_z$ therefore have the
same moments of every order, and the cited uniqueness result makes
them equal. 

Applying $t\mapsto|t|$ recovers the original measures,
so $\nu_{A,z}=\nu_z$. Substituting this identity into the known
weak convergence in probability proves
\begin{equation}\label{eq:shifted-singular-convergence}
 \nu_{n,z}\xrightarrow[n\to\infty]{\text{weakly in probability}}\nu_{A,z}
 \qquad\text{for every fixed }z\in\C.
\end{equation}

\bigskip

\paragraph{Step 2: logarithmic potentials and Hermitization.}
To pass from singular values to eigenvalues, we must classically control the behaviour of the logarithm near 0 and $\infty$. The most difficult estimate is for the lower-tail: 
\begin{equation}\label{eq:small-log-external-input}
 \lim_{R\to\infty}\limsup_{n\to\infty}
 \Prob\!\left(
   \int_{[0,e^{-R})}|\log (t)|\,\dd\nu_{n,z}(t)>\varepsilon
 \right)=0,
\end{equation}
for Lebesgue-almost every $z$ and every $\varepsilon>0$, and using the convention $|\log (0)|=\infty$.
This estimate was proved in \cite[Section~15]{sah2023limiting}. The upper tail is easier, because it follows directly from the second moment:
\[
 \E\int t^2\,\dd\nu_{n,z}(t)
 =\frac1n\E\|A_n-z\Id\|_{\mathrm{HS}}^2
 \le2d+2|z|^2.
\]
For $t>e^R$, $R>0$, we have $\log t\le e^{-R}t^2$. Thus
\[
 \E\int_{t>e^R}\log (t)\,\dd\nu_{n,z}(t)
 \le2e^{-R}(d+|z|^2).
\]
Together with Markov's inequality and
\eqref{eq:small-log-external-input}, this gives
\begin{equation}\label{eq:log-uniform-integrability}
 \lim_{R\to\infty}\limsup_{n\to\infty}
 \Prob\!\left(
  \int_{\{|\log (t)|>R\}}|\log (t)|\,\dd\nu_{n,z}(t)>\varepsilon
 \right)=0.
\end{equation}

We can now apply the Hermitization lemma
\cite[Lemma~4.3 and Remark~4.4]{bordenave2012around}.
Its hypotheses are precisely the deterministic singular-law limit
\eqref{eq:shifted-singular-convergence} and logarithmic uniform
integrability \eqref{eq:log-uniform-integrability}.
It yields weak convergence in probability of $\mu_{A_n}$ to a
probability measure $\mu$ satisfying
\[
 \int_{\C}\log|z-w|\,\dd\mu(w)
 =\int_0^\infty\log t\,\dd\nu_{A,z}(t)=L(z)
 \qquad\text{for Lebesgue-almost every }z.
\]
Consequently $\mu=(2\pi)^{-1}\Delta L=\mu_A$, as desired.
\end{proof}

\newpage

{\footnotesize

\bibliographystyle{plain}
\bibliography{refs}

}

\hrulefill

\end{document}